\documentclass[11pt]{amsart}
\usepackage{amsmath, amsthm}
\usepackage{amssymb}
\usepackage{amscd}
\usepackage{url}
\usepackage{graphicx}
\usepackage{color}

\usepackage{hyperref}

\numberwithin{equation}{section}

\theoremstyle{plain}

\newtheorem{theorem}[subsection]{Theorem}
\newtheorem{proposition}[subsection]{Proposition}
\newtheorem{lemma}[subsection]{Lemma}
\newtheorem{corollary}[subsection]{Corollary}

\theoremstyle{definition}

\newcommand{\Q}{\mathbb{Q}}
\newcommand{\Z}{\mathbb Z}
\newcommand{\C}{\mathbb C}

\newcommand{\R}{\mathbb{R}}
\newcommand{\lambdabold}{\text{\mathversion{bold}$\lambda$\mathversion{normal}}}
\newcommand{\mubold}{\text{\mathversion{bold}$\mu$\mathversion{normal}}}

\makeatletter
\def\author@andify{%
  \nxandlist {\unskip ,\penalty-1 \space\ignorespaces}%
    {\unskip {} \@@and~}%
    {\unskip \penalty-2 \space \@@and~}%
}
\makeatother

\title{Sums of squares in real quadratic fields and Hilbert modular groups}
\author{Fernando Chamizo}
\author{Roberto J. Miatello}
\thanks{The first named author is partially supported by the MTM2017-83496-P grant of the
MICINN (Spain) and by ``Severo Ochoa Programme for Centres of Excellence in
R{\&}D'' (SEV-2015-0554)}


\begin{document}

\begin{abstract}
We use the spectral theory of Hilbert-Maass forms for real quadratic fields 
to obtain the asymptotics of some sums involving the number of representations as a sum of two squares in the ring of integers. 
\end{abstract}

\maketitle

\section{Introduction}

The asymptotic study of the average of the arithmetical function $r(n)$ giving the number of representations  of~$n$ as a sum of two squares is the goal of the celebrated Gauss circle problem. It asks for the infimum of the exponents $\alpha$'s satisfying
\begin{equation}\label{circle}
 \sum_{n\le x}
 r(n)
 =
 \pi x
 +O(x^\alpha).
\end{equation}
The left-hand side counts the number of lattice points in a circle of radius $\sqrt{x}$ and $\pi x$ gives its area. Gauss used an approximation of this kind when studying the class number of quadratic forms \cite{gauss}. A simple geometric reasoning, already employed by Gauss, shows \eqref{circle} for $\alpha=1/2$. Sierpi\'nski \cite{sierpinski} proved the estimate for $\alpha=1/3$ with a rather complicated argument (see \cite{huxley_b} for a short modern elementary approach). On the other hand,  $\alpha=1/3$  can be obtained in a more direct way using the Euclidean spectral expansion (i.e., classical Fourier analysis) of radial functions on $\R^2$ \cite[VIII\S8]{Lang}.

A finer asymptotic property to be studied about $r(n)$ is its self-correlation. In \cite{iwaniec} it is proved that
\begin{equation}\label{corre}
 \sum_{n\le x}
 r(n)r(n+1)
 =
 8 x
 +O(x^{2/3})
\end{equation}
and there are similar formulas replacing $r(n+1)$ by $r(n+k)$ (see \cite{chamizo_c} for the uniformity). 
Although \eqref{circle} for $\alpha=1/3$  and \eqref{corre} seem unrelated, from the analytic point of view one can run both proofs along similar lines changing Euclidean spectral expansions based on Fourier series by hyperbolic spectral expansions based on Maass forms and Eisenstein series. This second situation is by far more involved and one has to bypass unsettled problems like the existence of exceptional eigenvalues or the $L^\infty$ bounds of the eigenfunctions, which become trivial in the Euclidean setting.  It is noteworthy to mention that the error term in \eqref{corre} remains unimproved while van der Corput method and other finer techniques of exponential sums have proved \eqref{circle} for some $\alpha<1/3$ \cite{huxley_b}, namely the best known result \cite{BoWa} allows to take $\alpha$ slightly smaller than~$0.314$.  

\

Several authors have considered the analogue of Gauss circle problem in totally real  number fields \cite{siegel}, \cite{schaal}, \cite{schaal_a}, \cite{rausch} where $n\in\Z$ is replaced by $\lambda\in\mathcal{O}$, with $\mathcal{O}$ the ring of integers. 
A basic issue is to find a real number field analogue of $n\le x$ (note that the existence of infinitely many units prevents us from using the norm). Following Siegel \cite{siegel}, a natural condition is to limit the size of every Galois conjugate. In the real quadratic case of discriminant $\Delta$, the main result in~\cite{schaal_a} implies that
\begin{equation}\label{schaa}
 \sum_{\substack{0\le \lambda<V_1\\ 0\le \lambda^\sigma<V_2}}
 r(\lambda)
 =
 \frac{\pi^2}{\Delta}V_1V_2+O\big((V_1V_2)^\alpha\big)
 \qquad\text{for any }\alpha>\frac 23
\end{equation}
where $\lambda^\sigma$ is the real conjugate of $\lambda$ and $r(\lambda)$ is defined in the natural way:
\begin{equation}\label{defr}
 r(\lambda)=
 \#\big\{ (\xi,\eta)\in\mathcal{O}^2\;:\; \lambda= \xi^2+\eta^2 \big\}.
\end{equation}

The purpose of this paper is to get an analogue of \eqref{corre} for real quadratic fields by applying the hyperbolic circle problem for products of two upper half-planes. This latter problem was studied in general in \cite{BrGrMi} for multiple products corresponding  to totally real number fields of arbitrary degree. The underlying analysis  involves the spectral theory of Hilbert-Maass forms.

To state our main results we use the abbreviations
\begin{equation}\label{def_N}
 N_D(V_1,V_2)
 =
 \sum_{\substack{0\le \lambda<V_1\\ 0\le \lambda^\sigma<V_2}}
 r(\lambda)r(\lambda+1)
 \qquad\text{and}\qquad
 C_D = 
 \frac{32\Delta}{ \big(2-\chi(2)+2\chi(4)\big) \sum_{n=1}^\Delta   n^2\chi(n) }
\end{equation}
where $\Delta$ is the discriminant of the quadratic real field $\Q(\sqrt{D})$ we are considering and $\chi$ is the character corresponding to the Kronecker symbol $\Big(\frac{\Delta}{\cdot}\Big)$. See the next section for more details about the notation. 

\medskip

Firstly we state a result like \eqref{schaa} for the self-correlation. 

\begin{theorem}\label{mainth1}
 For $V_1,V_2\to\infty$
 \[
  N_D(V_1,V_2)
  =
  C_D V_1V_2
  +O\big( (V_1V_2)^{3/4} \big).
 \]
\end{theorem}

Secondly, we are interested in the global uniformity in $V_1$ and $V_2$ that is not considered in~\cite{BrGrMi}. Note that there are limiting arithmetical situations, for instance if $V_2$ is like $V_1^{-1}$ we are essentially considering solutions of Pell's equation which have an exponential spacing, too sparse to be captured by harmonic analysis and by an asymptotic formula.

\begin{theorem}\label{mainth2}
 For $0<V_2<1$ and $V_1V_2^2\to\infty$, if there are no exceptional eigenvalues then
 \[
  N_D(V_1,V_2)
  =
  C_D V_1V_2
  +O\big( V_1^{3/4}V_2^{1/2} \big)
 \]
 and if there exist exceptional eigenvalues then we have to add 
 $O\big( V_1^{1/2+c}V_2^{1/4} \big)$ with $c=\sup \Im t_{\ell_ 1}$.  
\end{theorem}

{\sc Remark}. 
In principle one could suspect a chaotic behavior of $C_D$ because of the arithmetic nature of the character sum but it is not difficult to prove that $C_D\asymp D^{-3/2}$. A precise result is included in \S\ref{num_res}. 
It may sound surprising that the  existence or not  of exceptional eigenvalues plays a role in Theorem~\ref{mainth2} but not in Theorem~\ref{mainth1}. In the terminology  in \cite{BrGrMi}, we have a large spectral gap and the influence of the exceptional eigenvalues in Theorem~\ref{mainth1} is absorbed by the error term. The fundamental result here is the bound of Kim and Shahidi \cite{KiSh} for the size of the potential exceptional eigenvalues
that allows to take $c=1/9$ in Theorem~\ref{mainth2}. 

\section{Notation and basic concepts}

We follow mainly the notation of \cite{BrMi} and \cite{BrGrMi}.
We recall it briefly reviewing at the same time the basic concepts. We use Ladau's $O$-notation and Vinogradov's $\ll$-notation indistinctly.

\

The Poincar\'e half-plane $\mathbb{H}$ is the Riemannian manifold given by the upper half-plane $\Im z>0$ endowed with the hyperbolic metric
$y^{-2}(dx^2+dy^2)$, which induces the invariant measure $d\mu(z)=y^{-2}dxdy$. 
It is possible to give an explicit formula for the corresponding hyperbolic distance~$\rho$, namely
\begin{equation}\label{rhou}
 \cosh\rho(z,w)=1+2u(z,w)
 \qquad\text{with}\quad
 u(z,w)
 =
 \frac{|z-w|^2}{4\Im z\,\Im w}.
\end{equation}
The group 
$\textrm{PSL}_2(\R)=\textrm{SL}_2(\R)/\{\pm \textrm{Id}\}$
acts faithfully on $\mathbb{H}$ in the standard way by linear fractional transformations and in fact it coincides with the group of orientation preserving isometries of $\mathbb{H}$. This implies in particular that $u$ is invariant, meaning $u(z,w)=u\big(\gamma(z),\gamma(w)\big)$
for any $z,w\in \mathbb{H}$ and  $\gamma\in \textrm{PSL}_2(\R)$.

For each discrete subgroup $\Gamma < \textrm{PSL}_2(\R)$ such that $\Gamma\backslash \mathbb{H}$ has finite volume, the spectral theory of automorphic forms  allows to expand any 
$f\in L^2(\Gamma\backslash \mathbb{H})$
in terms of the eigenfunctions of the Laplace-Beltrami operator 
$\Delta =-y^2\big(\partial^2_x+\partial^2_y\big)$.
In some sense, 
the role of the Fourier transform is played in this context by the \emph{Selberg transform}
\[
 k\in C_0^\infty\big([0,\infty)\big)
 \longmapsto
 h(t)
 =
 \int_{\mathbb{H}}
 k\big(u(z,i)\big) y^{1/2+it}\;d\mu(z).
\]
As in the case of the Fourier transform, we can relax a lot the $C_0^\infty$ regularity still having a sound and useful Selberg transform. It is easier to introduce the conditions in terms of the transform itself, taking for granted the existence of the integral. Following Selberg \cite[Satz 3.4]{berg},  we ask for the existence of a strip $S_\delta=\{z\,:\,|\Im z|<1/2+\delta\}$ with $\delta >0$ such that
\begin{equation}\label{cond_h}
 h\text{ is holomorphic in $S_\delta$}
 \qquad\text{and}\qquad
 |h(z)|\ll (|z|+1)^{-2-\delta}
 \text{ for }z\in S_\delta.
\end{equation}
This is satisfied when $k\in C_0^\infty$ (see \cite[\S1.8]{iwaniec}).

A novelty with respect to the Euclidean setting is that in the cases of arithmetical relevance e.g., 
$\Gamma=\textrm{PSL}_2(\Z)$,
there is a discrete spectrum (corresponding mainly to Maass cusp forms) 
and a continuous spectrum (corresponding to Eisenstein series).
This non-classical harmonic analysis built with nonholomorphic automorphic forms has had a profound impact on analytic number theory specially  since the  development of Kuznetsov's formula \cite{hafner}. 

\

We focus on the case of real quadratic number fields $\Q(\sqrt{D})$ with $D\in\Z_{>1}$ squarefree. The corresponding ring of integers is 
\[
 \mathcal{O}=
 \begin{cases}
  \Z[\sqrt{D}] &\text{if } D\not\equiv 1\pmod{4},
  \\
  \Z\big[(1+\sqrt{D})/2\big] &\text{if } D\equiv 1\pmod{4}.
 \end{cases}
\]
To parallel the previous case $\textrm{PSL}_2(\Z)$,  the natural object to work with is the full Hilbert modular group 
\begin{equation}\label{hmg}
 \Gamma_{\mathcal{O}}=\big\{(\gamma, \gamma^\sigma)\,:\,\gamma\in \textrm{PSL}_2(\mathcal{O})\big\}
\end{equation}
acting on $\mathbb{H}^2= \mathbb{H}\times \mathbb{H}$ 
where $\gamma^\sigma$  denotes the action by the nontrivial element in the Galois group of $\Q(\sqrt{D})$ (the real conjugation) on the entries of $\gamma$. 
It turns out that $\Gamma_{\mathcal{O}}$ is a discrete subgroup of $\textrm{PSL}_2(\R)^2$
and $\Gamma_{\mathcal{O}}\backslash\mathbb{H}^2$ has finite volume. In general the groups with 
these two properties
are called \emph{lattices} and they are said to be \emph{irreducible} if the projections on each factor of  $\textrm{PSL}_2(\R)^2$ are dense. This avoids artificial examples like $\textrm{PSL}_2(\Z)^2$ that can be ``reduced'' to  discrete groups acting on $\mathbb{H}$. 

If $\Gamma$ is an irreducible lattice,  spectral theory allows to analyze $L^2(\Gamma\backslash\mathbb{H}^2)$ in terms of the simultaneous eigenfunctions 
$\psi=\psi(z_1,z_2)$
of $\Delta_{z_1}$ and $\Delta_{z_2}$ (where the subscript indicates the variable).
Imposing $\psi\in L^2(\Gamma\backslash\mathbb{H}^2)$ we have a discrete sequence of couples of eigenvalues  $\{\lambdabold_\ell\}_\ell$ with  $\lambdabold_\ell=(\lambda_{\ell_1}, \lambda_{\ell_2})$ and corresponding orthonormal eigenfunctions $\psi_\ell${, 
\[
 \Delta_{z_1}\psi_\ell=\lambda_{\ell_1}\psi_\ell,
 \qquad
 \Delta_{z_2}\psi_\ell=\lambda_{\ell_2}\psi_\ell.
\]}%
We reserve the label $\ell=0$ for the trivial couple $\lambdabold_0=(0,0)$
and consequently $\psi_0=|\Gamma\backslash\mathbb{H}^2|^{-1/2}$. 
It is said that~$\lambdabold_\ell$ is 
\emph{exceptional} 
if 
$0< \lambda_{\ell_1}<1/4$
or
$0< \lambda_{\ell_2}<1/4$,
and it is said to be 
\emph{totally exceptional} 
if both conditions hold simultaneously. 
The relevance of the exceptional $\lambdabold_\ell$  is that the analogue of the Fourier transform has a quite different behavior at them. 
To emphasize this point we write
\[
 \lambda_{\ell_ j}
 =
 \frac 14-t_{\ell_ j}^2
 =
 \Big(\frac 12 + it_{\ell_ j}\Big)
 \Big(\frac 12 - it_{\ell_ j}\Big)
 \qquad
 \text{with}\quad 
 t_{\ell_ j}\in [0,\infty)\cup i(0,1/2].
\]
(Note that we slightly divert  from \cite{BrGrMi}). In this way,  $t_{01}=t_{02}=i/2$ and $\lambdabold_\ell$ 
is exceptional if 
$\Im t_{\ell_ 1}$
or
$\Im t_{\ell_ 2}$
belong to $(0,1/2)$. 
Although it is conjectured that there are no exceptional 
$\lambdabold_\ell$ in the  cases of arithmetic interest (this is the generalization of a famous conjecture due to Selberg \cite{selberg}), in principle there might be infinitely many such 
$\lambdabold_\ell$. 
On the other hand, only finitely many  can be totally exceptional
(because the set
$\{\lambdabold_\ell\}$ is a discrete set) and the result of Kim and Shahidi \cite{KiSh} implies  
$\Im t_{\ell_ 1}, \Im t_{\ell_ 2}<1/9$
for the lattices $\Gamma$ in this
 paper. 

\

As in the one-dimensional case, 
it turns out that $\{\psi_\ell\}_\ell$ does not span 
$L^2(\Gamma\backslash\mathbb{H}^2)$
if $\Gamma\backslash\mathbb{H}^2$ is not compact and a continuous spectrum corresponding to Eisenstein series enters into the game. The corresponding spectral theorem is technical in nature and we only need a particular case, so we have limited its application to the proof of a single lemma. The reader preferring not to enter into the details of the 
proof can use Lemma~\ref{sp_th} as a black box embodying the spectral theorem. We refer the reader to \cite{BrGrMi} and \cite[Ch.1]{BrMi} for more extensive comments on the spectral theorem (see also \cite{harish} for a more comprehensive theory).

Rather than the expansion of functions in $L^2(\Gamma\backslash\mathbb{H}^2)$, we need to expand a type of automorphic kernels. To introduce them 
it is convenient to extend the definition of $u$ in \eqref{rhou} to $\mathbb{H}^2$ in the natural manner:
\[
 u(\mathbf{z},\mathbf{w})=
 \big(
 u(z_1,w_1),
 u(z_2,w_2)
 \big)
 \qquad\text{for}
 \quad
 \mathbf{z}=(z_1,z_2),\,\mathbf{w}=(w_1,w_2)
 \in 
 \mathbb{H}^2.
\]
Given an irreducible lattice $\Gamma$ and $k:[0,\infty)^2\longrightarrow\C$ decaying rapidly enough at zero and infinity, for instance $k$ compactly supported, we can construct an \emph{automorphic kernel}
\begin{equation}\label{aut_ker}
 K(\mathbf{z},\mathbf{w})
 =
 \sum_{\gamma\in\Gamma}
 k\big(u(\gamma(\mathbf{z}),\mathbf{w})\big).
\end{equation}
Using the fact that  $u$ is invariant by isometries, we deduce that $K$ is actually automorphic in both variables, that is
\[
K(\mathbf{z},\mathbf{w})=K(\gamma(\mathbf{z}),\mathbf{w})=K(\mathbf{z},\gamma(\mathbf{w}))
\qquad\text{for every $\gamma\in\Gamma$}.
\]

\section{An arithmetic lemma}

To study the sum of $r(\lambda)r(\lambda+1)$ by analytic methods it is convenient to consider more general weighted sums
\begin{equation}\label{N_k}
 \mathcal{N}_k
 =
 \sum_{\lambda\in\mathcal{O}}
 r(\lambda)r(\lambda+1)
 k(\lambda, \lambda^\sigma)
 \qquad\text{with}
 \quad
 k:\R^2\rightarrow\C.
\end{equation}

A preliminary consideration is whether this sum actually makes sense when $k$ decays rapidly enough. We state a general elementary result of this kind although, for the aims of this paper, we could restrict ourselves to compactly supported functions. 

\begin{lemma}\label{conver}
 If 
 $k(x,y)\ll \big(x+y+1)^{-\alpha}$ with $\alpha>3$ for $x,y\ge 0$ then $\mathcal{N}_k$ is well-defined. 
\end{lemma}

\begin{proof}
 Note first that if $\lambda$ is a sum of two squares then so  is $\lambda^\sigma$ and if both are positive then we can restrict the sum to $\lambda,\lambda^\sigma> 0$.
 
 If $\lambda=n+m\sqrt{D}$, expanding $\lambda=(a+b\sqrt{D})^2+(c+d\sqrt{D})^2$ we see that $r(\lambda)$ counts the number of integral or half-integral solutions of
 \[
  \begin{cases}
   a^2+c^2+D(b^2+d^2)=n,
   \\
   2ab+2cd=m.
  \end{cases}
 \]
 The positivity of the first equation shows at once that $r(\lambda)$ is well-defined i.e., $r(\lambda)<\infty$. 
 In fact using that the number of representations of an integer as a sum of two squares in $\Z$ tends to zero when divided by any positive power \cite[Th.\,338]{HaWr}, we have the trivial bound
 $r(\lambda)=O(n^{1+\epsilon})$ for any $\epsilon>0$. 
 Note also that, necessarily, in order to have a solution one must have $|m|<4n$.
 
 The inequality $2r(\lambda)r(\lambda+1)\le r^2(\lambda)+r^2(\lambda+1)$ and the equation $\lambda+\lambda^\sigma= 2 n$ reduce the assertion to proving that
 \[
   \mathop{\sum\!\sum}_{0\le m< 4n}
   \frac{r^2(n+m\sqrt{D})}{n^\alpha}
   <\infty
   \qquad\text{for}\quad\alpha>3.
 \]
 Using the trivial bound we have
 \[
   \mathop{\sum\!\sum}_{0\le m< 4n}
   \frac{r^2(n+m\sqrt{D})}{n^\alpha}
   \ll
   \mathop{\sum\!\sum}_{0\le m< 4n}
   \frac{r(n+m\sqrt{D})}{n^{\alpha-1-\epsilon}}
   \ll
   \sideset{}{'}\sum_{a,b,c,d}
   \big(a^2+Db^2+c^2+Dd^2\big)^{1+\epsilon-\alpha}
 \]
where we disregard the value $a=b=c=d=0$ in the last sum. It is plain that the latter series converges when $\alpha>3+\epsilon$, for instance by comparing with the integral $\int_{B'} \|\vec{x}\|^{-2-\delta}d\vec{x}$, $\delta>0$, where $B'$ is the exterior of the unit ball in $\R^4$. 
\end{proof}

\

The key point to apply spectral methods is to translate $\mathcal{N}_k$ into an automorphic kernel \eqref{aut_ker}. The argument is an adaptation of that in the 1-dimensional case in \cite[Cor.12.2]{iwaniec}.

\begin{lemma}\label{alemma}
 Consider the lattice in $\textrm{\rm PSL}_2(\R)^2$ defined as
 \[
 \Gamma=
 \big\{
 (
 \gamma,\gamma^\sigma
 )
 \;:\; \gamma\in M/\{\pm\text{\rm Id}\}
 \big\}
 \quad\text{with}\quad
 M=
 \left\{
 \begin{pmatrix}
  a
  &
  b
  \\
  c
  &
  d
 \end{pmatrix}
 \in\textrm{\rm SL}_2(\mathcal{O})
 \,:\,
 a+d, b+c\in 2\mathcal{O}
 \right\}.
 \]
 Then for $k$ as in Lemma~\ref{conver} we have
 \[
  \mathcal{N}_k
  =
  2
  \sum_{\gamma\in \Gamma}
  k\big(u(\gamma (\mathbf{i}),\mathbf{i})\big)
  \qquad\text{where}\quad \mathbf{i}=(i,i).
 \]
\end{lemma}

\begin{proof} 
The map
\begin{eqnarray*}
 \mathcal{C}:=
 \big\{
 (A,B,C,D)\in \mathcal{O}^4
 \,:\,
 A^2+B^2=C^2+D^2+1
 \big\}
 &\longrightarrow&
 M
 \\
 (A,B,C,D)
 &\longmapsto&
 \begin{pmatrix}
  A+C& B+D
  \\
  D-B& A-C
 \end{pmatrix}
\end{eqnarray*}
clearly establishes a bijection between $\mathcal{C}$ and $M$. On the other hand, if $\tau$ denotes the last matrix a calculation proves  $u\big(\tau(i),i\big)=C^2+D^2$. Hence  
\[
  \mathcal{N}_k
  =
  \sum_{(A,B,C,D)\in\mathcal{C}}
  k\big(C^2+D^2,(C^2+D^2)^\sigma\big)
  =
  \sum_{\tau\in M}
  k\big(u(\tau (i),i), u(\tau^\sigma (i),i)\big).
\]
This proves the result because $\pm\tau$ give rise to the same element in $\Gamma$.
Note that the sign changes do not affect the values of $\lambda=u(\tau (i),i)$ and $\lambda^\sigma=u(\tau^\sigma (i),i)$.
\end{proof}

\section{A rough spectral bound}

Here we state the consequence of the application of the spectral theorem to automorphic kernels in the form needed for our purposes. It will be convenient to classify the labels of the exceptional~$\lambdabold_\ell$ in three sets:
\[
 \Lambda_0
 =
 \big\{
 \ell\,:\, \lambdabold_\ell\text{ totally exceptional}
 \big\},
 \qquad
 \Lambda_j
 =
 \big\{
 \ell\not\in\Lambda_0\,:\, \Im t_{\ell j}\in (0,1/2)
 \big\}
 \quad j=1,2.
\]

\begin{lemma}\label{sp_th}
 Let $k_1$ and $k_2$ be continuous functions $k_j:[0,\infty)\longrightarrow\C$ with Selberg transforms $h_j$ satisfying \eqref{cond_h}. 
 Consider the automorphic kernel \eqref{aut_ker} with $k(x,y)=k_1(x)k_2(y)$. Define
 \[
  H_j
  =
  \sum_{n=0}^\infty
  2^{2n}\sup_{t\in I_n}|h_j(t)|
  \qquad\text{with}\quad
  I_0=[0,2)\text{ and }
  I_n=[2^n,2^{n+1})
  \text{ for }n\ge 1
 \]
 and
  \begin{equation} \label{eq:M}
  \mathcal{M}
  =
  \frac{h_1(i/2)h_2(i/2)}{|\Gamma\backslash\mathbb{H}^2|}
  +
  \sum_{\ell\in\Lambda_0}
  h_1(t_{\ell 1})h_2(t_{\ell 2})
  \overline{\psi}_\ell(\mathbf{z})
 {\psi}_\ell(\mathbf{w}).
 \end{equation}
 Then we have
 \[
  K(\mathbf{z},\mathbf{w})
  =
  \mathcal{M}
  +O_{\mathbf{z}, \mathbf{w},\Gamma}\big(
  H_1H_2
  +H_1 \sup_{\ell\in\Lambda_2}|h_2(t_{\ell_ 2})|
  +H_2 \sup_{\ell\in\Lambda_1}|h_1(t_{\ell_ 1})|
  \big).
 \]
% where the implicit $O$ constant depends on $\mathbf{z}$, $\mathbf{w}$ and $\Gamma$.  
\end{lemma}

\noindent{\sc Remark}. 
The dependence of the $O$ constant on  $\mathbf{z}$ and $\mathbf{w}$ could be made explicit but it is irrelevant in our application.
Of course if $\Lambda_0=\emptyset$ the sum over $\ell\in \Lambda_0$ must be omitted in $\mathcal{M}$ and if there are no exceptional eigenvalues, the same applies to the suprema over $\Lambda_1$ and $\Lambda_2$ in the error term. 

\begin{proof}
 The spectral expansion of $K$, the analogue of the Poisson summation formula,  as given in
\cite[(39)]{BrGrMi}, reads
\begin{multline*}
 K(\mathbf{z},\mathbf{w})
 =
 \sum_{\ell} h(\mathbf{t}_{\ell})
 \overline{\psi}_\ell(\mathbf{z})
 {\psi}_\ell(\mathbf{w})
 +
 \\
 +
 2\sum_{\kappa}c_\kappa
 \sum_{\mu\in \mathcal{L}_\kappa}
 \int_{(\R^+)^2}
 h(\mathbf{t}+\mubold)
 \overline{E}(\kappa;i\mathbf{t},i\mubold;\mathbf{z})
 {E}(\kappa;i\mathbf{t},i\mubold;\mathbf{w})
 \; dt_1 dt_2
\end{multline*}
where $h(t_1,t_2)=h_1(t_1)h_2(t_2)$, $\kappa$ runs over the finitely many  inequivalent cusps, $c_\kappa$ are positive constants, $\mathcal{L}_\kappa$ is a  lattice in $\R^2$ and $E$ denotes the Eisenstein series. 
In the first sum the terms with $\ell \in \Lambda_0\cup\{0\}$ contribute exactly as $\mathcal{M}$. 
Let $K^*=K-\mathcal{M}$, we have to prove that it is bounded by 
the error term in the statement. 
Using that $|ab|\le (|a|^2+|b|^2)/2$ we have for $\mathbf{v}=\mathbf{z}$ or~$\mathbf{v}=\mathbf{w}$
{\begin{multline}\label{Kstar}
 |K^*(\mathbf{z},\mathbf{w})|
 \le
 \sum_{\ell\not\in\Lambda_0\cup\{0\}} |h(\mathbf{t}_{\ell})|
 |{\psi}_\ell(\mathbf{v})|^2
 +
 \\
 +
 2\sum_{\kappa}c_\kappa
 \sum_{\mu\in \mathcal{L}_\kappa}
 \int_{(\R^+)^2}
 |h(\mathbf{t}+\mubold)|
 |{E}(\kappa;i\mathbf{t},i\mubold;\mathbf{v})|^2
 \; dt_1 dt_2.
\end{multline}}

Now we need a form of Bessel's inequality that allows to bound, for a fixed $\mathbf{z}\in\mathbb{H}^2$ and every $(n_1,n_2)\in\Z_{\ge 0}^2$, the expression
\[
 S(n_1,n_2,\mathbf{z})
 =
 \sum_{\ell\in \mathcal{X}_d} 
 |{\psi}_\ell(\mathbf{z})|^2
 +
 2\sum_{\kappa}c_\kappa
 \sum_{\mu\in \mathcal{L}_\kappa}
 \int_{\mathcal{X}_c}
 |{E}(\kappa;i\mathbf{t},i\mubold;\mathbf{z})|^2
 \; dt_1 dt_2
\]
where
\[
 \mathcal{X}_d=
 \big\{\ell\,:\,
 t_{\ell j}
 \in I_{n_j}\cup i(0,1/2]
 \big\}
 \qquad\text{and}\qquad
 \mathcal{X}_c=
 \big\{\mathbf{t}\in(\R^+)^2\,:\,
 \pm(t_j+\mu_j)
 \in I_{n_j}
 \big\}.
\]
The instance of Bessel's inequality we need is 
\cite[Th.\,4.2]{BrGrMi}
\begin{equation}\label{bessel_i}
 S(n_1,n_2,\mathbf{z})
 =
 O\big(
 2^{2n_1+2n_2}
 \big).
\end{equation}
The intuitive interpretation is that $\psi_\ell$ and $E$ behave as constants on average.

If we divide the integral in \eqref{Kstar} into the dyadic pieces $\pm (t_j+\mu_j)\in I_{n_j}$ indicated by $\mathcal{X}_c$ and, using the positivity, we apply \eqref{bessel_i} to each of them, we have that the last term in \eqref{Kstar} contributes at most
\[
 2
 \sum_{n_1,n_2=0}^\infty
  \sup_{t\in I_{n_1}}|h_1(t)|
  \sup_{t\in I_{n_2}}|h_2(t)|
 \sum_{\kappa}c_\kappa
 \sum_{\mu\in \mathcal{L}_\kappa}
 \int_{\mathcal{X}_c}
 |{E}(\kappa;i\mathbf{t},i\mubold;\mathbf{v})|^2
 \; dt_1 dt_2
 =
 O(H_1H_2).
\]
The same argument works to get this bound for the contribution of the first term in the right-hand side of \eqref{Kstar} when $t_{\ell 1}$ and $t_{\ell 2}$ are real. The remaining terms have $\ell \in \Lambda_1\cup \Lambda_2$ and we can proceed in the same way keeping the supremum of $h_j$ if $\ell \in \Lambda_j$. For instance, the terms with  $\ell \in \Lambda_1$ contribute at most
\[
 \sup_{t_{\ell_ 1}}|h_1(t_{\ell_ 1})|
 \sum_{\ell\in\Lambda_1} 
 \sum_{n_2=0}^\infty
 \sup_{t\in I_{n_2}}|h_2(t)|
 |{\psi}_\ell(\mathbf{v})|^2
 \le
 \sup_{t_{\ell_ 1}}|h_1(t_{\ell_ 1})|
 \sum_{n_2=0}^\infty
 \sup_{t\in I_{n_2}}|h_2(t)|
 S(0,n_2,\mathbf{v})
\]
and the sum is $O(H_2)$ by \eqref{bessel_i}. 

Therefore, we have proved that $|K^*(\mathbf{z},\mathbf{w})|$ is bounded by the error  term appearing in the statement.
\end{proof}

\section{Volume computations}
The main term \eqref{eq:M} in the spectral expansion depends on the volume of the fundamental region and it becomes closely related to the constant $C_D$ appearing in the asymptotic formulas in Theorem~\ref{mainth1} and Theorem~\ref{mainth2}.

Our aim in this section is to prove the following result:

\begin{proposition}\label{pvolume}
 Let $\Gamma$ be the subgroup of the full Hilbert modular group  
 $\Gamma_\mathcal{O}$
 introduced in Lemma~\ref{alemma}. The volume of a fundamental region of $\Gamma\backslash\mathbb{H}^2$ is given by
 \[
  |\Gamma\backslash\mathbb{H}^2|
  =
  \big(2-\chi(2)+2\chi(4)\big)
  \frac{\pi^2}{\Delta}
  \sum_{n=1}^\Delta
  n^2\chi(n)
 \]
 where $\chi$ is the character corresponding to the Kronecker symbol $\Big(\frac{\Delta}{\cdot}\Big)$ and $\Delta$ is the discriminant of $\Q(\sqrt{D})$ i.e., $\Delta=D$ if $4\mid D-1$ and $\Delta=4D$ otherwise.  
\end{proposition}

{To prove it we are going to use an old result due to Siegel \cite{siegel2} and the computation of} the index $[\Gamma_\mathcal{O}:\Gamma]$. We give indeed an explicit description of the representatives of the subgroups.

The elements of the full Hilbert modular group and its subgroups are pairs of matrices related by the real conjugation. This redundant presentation is important when the action on $\mathbb{H}^2$ is considered but from the group theoretical point of view we get an isomorphic group dropping the second matrix in the pair. For the sake of brevity, in the {next result and in the} rest of the section we identify $(\gamma, \gamma^\sigma)$ and $\gamma$ when the action on $\mathbb{H}^2$ is irrelevant.

\begin{proposition}\label{cosets}
 Define
 \[
  T_u=
  \begin{pmatrix}
   1& u
   \\
   0&1
  \end{pmatrix},
  \ \ \ 
  S=
  \begin{pmatrix}
   0& 1
   \\
   -1&0
  \end{pmatrix},
  \ \ \ 
  \omega = \frac{1+\sqrt{D}}{2},
  \ \ \ 
  \overline{\omega} = \frac{1-\sqrt{D}}{2},
  \ \ \ 
  \eta=
  \begin{cases}
   1+\sqrt{D} &\text{if } 4\mid D-3,
   \\
   \sqrt{D} &\text{if } 4\mid D-2
  \end{cases}
 \]
 and 
 \[
  C_1=\{\textrm{\rm Id}\}\cup\{T_u\}_{u\in \Omega},
  \quad
  C_2=\{ST_u\}_{u\in \Omega},
  \quad
  C_3=\{T_uST_v\}^{u\in \Omega}_{v\in \Omega^*}
  ,\quad
  {C_4}=\{ST_vST_v\}_{v\in \Omega^*}
 \]
 where 
 $\Omega =\{1,\omega,\overline{\omega}\}$
 and 
 $\Omega^* =\{\omega,\overline{\omega}\}$.
 Then a complete set of representatives of the cosets~$\Gamma_{\mathcal{O}}/\Gamma$ is given by
 \[
  \mathcal{R}_D
  =
  \begin{cases}
   \big\{\text{\rm Id}, T_1,T_\eta, T_{\eta+1},ST_1, ST_{\eta+1}\big\}
   &\text{if }4\nmid D-1,
   \\
   C_1\cup C_2\cup \big\{T_1ST_{\omega}, T_1ST_{\overline{\omega}}\big\}
   &\text{if }8\mid D-1,
   \\
   C_1\cup C_2\cup C_3\cup C_4
   &\text{if }8\mid D-5.
  \end{cases}
 \]
\end{proposition}

{Hence the index in these three cases is respectively, 6, 9 and 15. This can be written in an artificial but compact way with the Kronecker symbol.}

\begin{corollary}\label{c_index}
 We have 
 $
  [\Gamma_{\mathcal{O}}:\Gamma]
  =
  6-3\chi(2)+6\chi(4)
 $
 with $\chi$ as in Proposition~\ref{pvolume}.
\end{corollary}

{We divide the proof of Proposition~\ref{cosets} according whether $8$ divides $D-5$ or not.  In this second case we benefit from a simple description  of the group $\Gamma$ given in the following result which allows a substantial reduction in the computations. In its proof and in that of Proposition~\ref{cosets} we will use}
that for $4\nmid D-1$, we have $\eta^2\in 2\mathcal{O}$ and {when writing} each $x\in \mathcal{O}$ as $x=x_1+x_2\eta$, $x_1,x_2\in\Z$ we have $x\in 2\mathcal{O}$ if and only if $x_1$ and $x_2$ are even. 

\begin{lemma}\label{conjugation}
 If $8\nmid D-5$ we have $\Gamma =C^{-1}\Gamma_0(2\mathcal{O})C$ where $C=ST_1$
 and $\Gamma_0(2\mathcal{O})$ is the subgroup of matrices in $\Gamma_{\mathcal{O}}$ with lower left entry in $2\mathcal{O}$. 
 Moreover, two matrices in $\Gamma_{\mathcal{O}}$ belong to the same left coset of $\Gamma_0(2\mathcal{O})$ if and only if the determinant of the matrix formed by their first columns belongs to $2\mathcal{O}$.
\end{lemma}

\begin{proof}
Note the computations
\[
 C
 \begin{pmatrix}
  a&b\\ c&d
 \end{pmatrix}
 C^{-1}
 =
 \begin{pmatrix}
  *&*\\ a+c-b-d&*
 \end{pmatrix}
 \quad\text{and}\quad
 C^{-1}
 \begin{pmatrix}
  a&b\\ 2c&d
 \end{pmatrix}
 C
 =
 \begin{pmatrix}
  b+d& -a+b-2c+d\\ -b& a-b
 \end{pmatrix}.
\]
The first one shows 
$\Gamma \subset C^{-1}\Gamma_0(2\mathcal{O})C$
because if $a+d, b+c\in 2\mathcal{O}$
then $a+c-b-d\in 2\mathcal{O}$.

The second computation shows that $\Gamma \supset C^{-1}\Gamma_0(2\mathcal{O})C$, it reduces 
to check that 
$ad-1\in 2\mathcal{O}$
implies 
$a+d\in 2\mathcal{O}$.
We consider two cases depending whether or not $8$ divides $D-1$.

If $8\mid D-1$, write $a=a_1+a_2\omega$ and $d=d_1+d_2\omega$ and note 
$\omega^2-\omega\in 2\mathcal{O}$.
Expanding $ad$ we get that 
$ad-1\in 2\mathcal{O}$
if and only if $2\nmid a_1d_1$ and $2\mid a_1d_2+a_2d_1+a_2d_2$
or equivalently if 
$a_1$, $d_1$ are both odd and $a_2$, $d_2$ are both even. Hence  $a+d\in 2\mathcal{O}$.

If $8\nmid D-1$, write $a=a_1+a_2\eta$ and $d=d_1+d_2\eta$.
Expanding
$ad$ we see that
$ad-1\in 2\mathcal{O}$
implies 
$2\nmid a_1d_1+a_2d_2$
and
$2\mid a_1d_2+a_2d_1$
in particular $a_1$ and $a_2$ have different parity.
In fact we can assume $2\nmid a_1$ (by the symmetry 
$a_1\leftrightarrow a_2$, $d_1\leftrightarrow d_2$) 
then $2\mid a_2$ and we conclude $2\nmid d_1$, $2\mid d_2$ that gives 
$a+d\in 2\mathcal{O}$ as expected.

Finally, $\gamma_1,\gamma_2\in \Gamma_{\mathcal{O}}$ belong to the same coset if and only  if $\gamma_2^{-1}\gamma_1\in \Gamma_0(2\mathcal{O})$
and the last part of assertion in the result reduces to write the formula for the lower left entry of this product.
\end{proof}

\begin{proof}[of Proposition~\ref{cosets} for $8\nmid D-5$]
 We check first that different elements in  $\mathcal{R}_D$ represent different cosets. 
 A calculation shows that $T_u^{-1}ST_v\in \Gamma$ for $u,v\in\mathcal{O}$ if and only if $u-v, v^2\in 2\mathcal{O}$. Then $T_u$ and $ST_v$ are in different cosets when $v\in \{1,\eta+1\}$ if $4\nmid D-1$ and when $v\in \{1,\omega,\overline{\omega}\}$ if $8\mid D-1$. 
 Clearly $T_u$ and $T_v$ belong to different cosets when 
 $u-v\not\in 2\mathcal{O}$, since $T_u^{-1}T_v=T_{v-u}$, and the same applies to $ST_u$ and $ST_v$.
 It only remains to check that in the case  $8\mid D-1$ the elements 
 $T_1ST_\omega$
 and
 $T_1ST_{\overline{\omega}}$
 do not share coset with the other elements. 
 As $T_u$, $ST_\omega$ and $ST_{\overline{\omega}}$ are in different cosets, the same holds for $T_u$ and $T_1ST_v$ with $v\in\{\omega, \overline{\omega}\}$. 
 Writing $u=a+bv$ and using $v^2-v\in 2\mathcal{O}$, after a calculation  $(ST_u)^{-1}T_1ST_v\in \Gamma$ imposes  $2\mid a+1$, $2\mid b+1$ and $2\mid a+b+1$
 which leads to a contradiction.

 \medskip
 
 We focus firstly on the case $4\nmid D-1$. 
 By Lemma~\ref{conjugation} we have to prove that for each element in $\Gamma_\mathcal{O}$ 
 there is an element in $C\mathcal{R}_DC^{-1}$ belonging to the same coset of $\Gamma_0(2\mathcal{O})$. 
 The first column of the matrices in $C\mathcal{R}_DC^{-1}$ is given by the following vectors, except for adding to the coordinates elements of $2\mathcal{O}$,
 \[
  \text{\rm Id}\to 
  \begin{pmatrix}
   1
   \\
   0
  \end{pmatrix}
  ,\ \ 
  T_1\to 
  \begin{pmatrix}
   1
   \\
   1
  \end{pmatrix}
  ,\ \ 
  T_\eta\to 
  \begin{pmatrix}
   1
   \\
   \eta
  \end{pmatrix}
  ,\ \ 
  T_{\eta+1}\to 
  \begin{pmatrix}
   1
   \\
   \eta+1
  \end{pmatrix}
  ,\ \ 
  ST_{1}\to 
  \begin{pmatrix}
   0
   \\
   1
  \end{pmatrix}
  ,\ \ 
  ST_{\eta+1}\to 
  \begin{pmatrix}
   \eta
   \\
   \eta+1
  \end{pmatrix}.
 \]
 With the criterion given at the end of Lemma~\ref{conjugation} it is enough to prove that if we add to these vectors a column with elements $a,b\in\mathcal{O}$ at least one of the corresponding determinants is in $2\mathcal{O}$. 
 Clearly we can assume $a,b\in \{0,1,\eta,\eta+1\}$.
 Note that the vectors corresponding to $\text{\rm Id}$ and $ST_1$ take care of all the cases with $a$ or $b$ zero. Then there are nine cases to be considered. Three of them have $a=b$ and the determinant with the second vector is zero. By the same reason, we can also disregard the three cases  in which $a$ and $b$ form the vectors corresponding to $T_\eta$, $T_{\eta+1}$ and $ST_{\eta+1}$.  The remaining cases are $(a,b)=(\eta+1,\eta)$, $(\eta+1,1)$ and $(\eta,1)$ and we get determinants in $2\mathcal{O}$ using the vectors corresponding respectively to $T_\eta$, $T_{\eta+1}$ and $ST_{\eta+1}$.
 
 \medskip
 
 We deal now with the case {$8\mid D-1$}. 
 To prove that for each element in $\Gamma_\mathcal{O}$ 
 there exists an element in $C\mathcal{R}_DC^{-1}$ in the same coset of $\Gamma_0(2\mathcal{O})$ we proceed as before. This case is much simpler and no calculations are needed because after excluding the cases $b=0$, $a=0$ and $a=b$ using as above respectively $\text{\rm Id}$, $ST_1$ and $T_1$, we have only six possibilities with  $a,b\in \{1,\omega,\overline{\omega}\}$ and they are all covered since each remaining element in $C\mathcal{R}_DC^{-1}$ treats at least the case corresponding to its own first column. 
\end{proof}

If $8\mid D-5$ the group $\Gamma$ is not a conjugate of $\Gamma_0(2\mathcal{O})$. In this case, and actually also when $8\mid D-1$, there is a simple set of generators of $\Gamma_{\mathcal{O}}$. Considering some relations among them it is possible to simplify any word to one of the representatives indicated below multiplied by an element of $\Gamma$. The drawback of this method is that it leads to distinguish a number of cases that require somewhat tedious calculations.
{ The advantage is that it gives a unified treatment of the case $4\mid D-1$ (see the remarks after the proof) and it potentially works for other subgroups.}

\begin{proof}[of Proposition~\ref{cosets} for $8\mid D-5$]
 A result due to Vaser\v{s}te\u{\i}n \cite{vaserstein}  (see also \cite{liehl}, \cite[\S1.2.2]{mayer} and \cite[\S5.1]{everhart})
 assures that the group $\Gamma_\mathcal{O}$ is generated by $S$, $T_1$ and $T_\omega$. 
 We note that $S^2 =-\textrm{Id}$ and $S\in \Gamma$. It is clear that, modulo $\Gamma$ we can write any element $g\in \Gamma_{\mathcal O}$ as an alternating product of factors equal to  $S$ and $T_u$ with $u \in  
{\mathcal O}$, that we will call generically a word. We employ the usual notation $g_1\sim g_2$, or $g_1$ and $g_2$ equivalent modulo $\Gamma$, to mean that $g_1,g_2\in \Gamma_{\mathcal{O}}$ belong to the same left coset i.e., $g_2^{-1}g_1\in \Gamma$. 
Note that $gS\sim g$ and then we can always consider words with $T_u\not\in\Gamma$ to the right.

 By the multilinear properties of the multiplication of matrices, replacing in a product $T_u$ by $T_{u+w}$ with $w$ in $2\mathcal{O}$  changes the entries of the final result in elements in this ideal. Recalling the definition of $\Gamma$, it is easy to see that two matrices in $\Gamma_{\mathcal{O}}$ with entries differing in  elements of $2\mathcal{O}$ belong to the same coset. Hence it is enough to consider products involving translations $T_u$ with $u\in \Omega$ because for any $u\in \mathcal{O}$ we can find $w\in 2\mathcal{O}$  such that $u-w\in \{0\}\cup\Omega$.
 Note that any word of length one is equivalent to exactly a one element of $C_1$.
 
 \medskip
 
 A calculation shows
 \begin{equation}\label{mrel1}
  T_v^{-1}ST_uST_v =
  \begin{pmatrix}
  -1-uv& -uv^2\\u& uv-1 
  \end{pmatrix}.
 \end{equation}
 For $v=1$ it lies in $\Gamma$. This implies that for any $g \in \Gamma_{\mathcal O}$, 
 \begin{equation}\label{red1}
  gT_uST_1 = gST_1 (T_{-1}ST_uST_1) \sim gST_1.
 \end{equation}
 Hence any element of the  form $gT_1\not\in\Gamma$ is equivalent to $T_1$ or $ST_1$.
 
 The set $C_2$ contains representatives of all the words of length 2 and they are clear nonequivalent. We have
 \begin{equation}\label{mrel2}
  T_uST_v =
  \begin{pmatrix}
  -u& 1-uv\\ -1& -v 
  \end{pmatrix}
 \end{equation}
 and it is easy to deduce that $C_1\cup C_2$ does not contain equivalent elements. As shown before, the words $T_uST_v$ can be simplified if $v=1$, hence they are all equivalent to some element in  $C_1\cup C_2\cup C_3$.  Noticing
 \begin{equation}\label{om_sq}
  \omega\overline{\omega}-1\in 2\mathcal{O}
  \qquad\text{and}\qquad
  \omega^2-\overline{\omega}\in 2\mathcal{O},
 \end{equation}
 it follows from \eqref{mrel2} that $C_1\cup C_3$ does not contain equivalent elements and by \eqref{mrel1} (recall $S^{-1}=-S$) any element in $C_2$ is not equivalent to an element in $C_3$. 
 
 For the words of length four $ST_uST_v$ we could consider in principle $u\in\Omega$, $v\in\Omega^*$ which makes six possibilities but some calculations show that $(T_{\overline{\omega}}S T_{\omega})^{-1}
  ST_{\overline{\omega}}S T_{\omega}$ and 
  $(T_{1}S T_{\overline{\omega}})^{-1}
  ST_{1}S T_{\omega}$ are respectively
 \[
  \begin{pmatrix}
   \omega +\overline{\omega}(\omega \overline{\omega}-1)
   &
   \omega^2+(\omega \overline{\omega}-1)^2
   \\
   -1-\overline{\omega}^2
   &
   -\omega -\overline{\omega}(\omega \overline{\omega}-1)
  \end{pmatrix}
  \quad\text{and}\quad
  \begin{pmatrix}
   2\overline{\omega}-1
   &
   2\omega\overline{\omega}-\omega-\overline{\omega}+1
   \\
   -2
   &
   1-2\omega
  \end{pmatrix},
 \]
 that belong to $\Gamma$ using \eqref{om_sq}. Taking conjugates we deduce
 \begin{equation}\label{mrel3}
  ST_{\overline{\omega}}S T_{\omega} \sim T_{\overline{\omega}}S T_{\omega},
  \quad
  ST_{{\omega}}S T_{\overline{\omega}} \sim T_{{\omega}}S T_{\overline{\omega}},
  \quad
  ST_{1}S T_{{\omega}} \sim T_{1}S T_{\overline{\omega}},
  \quad
  ST_{1}S T_{\overline{\omega}} \sim T_{1}S T_{{\omega}}.
 \end{equation}
 Then the only possibilities for elements not equivalent to those in $C_1\cup C_2\cup C_3$ are the ones considered in $C_4$.
 
 \medskip
 
 Assuming that the elements in $\mathcal{R}_D$ are not equivalent, we are going to prove that they form a complete set of representatives. This follows by an inductive argument if we prove that the words with length~5 are equivalent to shorter words because we have seen that the words of length at most~4 are equivalent to elements in $\mathcal{R}_D$. By \eqref{mrel3} and \eqref{red1} it is enough to consider words with 
 $ST_\omega ST_\omega$
 or
 $ST_{\overline{\omega}} ST_{\overline{\omega}}$
 to the right.
 In fact, taking real conjugates if necessary, it is enough to prove that 
 $T_uST_\omega ST_\omega\sim  ST_\omega ST_\omega$, or equivalently
 \begin{equation}\label{mrel4}
  (ST_\omega ST_\omega)^{-1}
  T_uST_\omega ST_\omega
  =
   \begin{pmatrix}
   u(\omega^2-1)\omega +\omega-\omega^2+1 
   &
   u(\omega^2-1)^2
   \\
   -u\omega^2
   & 
   -u \omega (\omega^2-1)  + 1 
   \end{pmatrix}
   \in\Gamma.   
 \end{equation}
 Using \eqref{om_sq} we have that $\omega-\omega^2+1\in 2\mathcal{O}$ and then the trace of this matrix belongs to $2\mathcal{O}$. In the same way, $(\omega^2-1)^2-\omega^2 =\omega^4-3\omega^2+1$
 differs from $\omega-3\overline{\omega}+1\in 2\mathcal{O}$ in an element of $2\mathcal{O}$. Hence the matrix is in $\Gamma$. 
 
 \medskip
 
 It remains to prove that $\mathcal{R}_D$ does not contain equivalent elements. 
 Recall that we knew that this was the case for $C_1\cup C_2\cup C_3$.
 Note that
 $(ST_{\overline{\omega}} ST_{\overline{\omega}})^{-1}
  ST_\omega ST_\omega
  =T_{\overline{\omega}} ST_1ST_\omega$ 
  which is not in~$\Gamma$ by \eqref{mrel1}. 
  Then 
  $ST_\omega ST_\omega\not\sim ST_{\overline{\omega}} ST_{\overline{\omega}}$. By real conjugation we can restrict ourselves to check that $g= ST_\omega ST_\omega$ is not equivalent to any any element in
  $C_1\cup C_2\cup C_3$.
  From \eqref{mrel4}, $T_ug\sim g$ and then $g$ is not equivalent to any element in $C_1\cup C_3$. Clearly $h^{-1}g\in C_2\cup C_3$ for $h\in C_2$, then $g\not \sim h$ also in this case and the proof is complete. 
\end{proof}

{
As mentioned before, it is possible to extend the previous proof to cover the case $8\mid D-1$. We now sketch the alternative proof
following the same lines as in the case when $8| D-5$.

The calculation \eqref{mrel1} does not depend on the properties of $D$ and we can deduce as before that $gT_1$ is equivalent to $\textrm{Id}$, $T_1$ or $ST_1$. In fact \eqref{mrel1} implies
\begin{equation}\label{eq:k-even}
 T_uST_1 \sim ST_1,
 \qquad 
 T_\omega S T_\omega \sim ST_\omega
 \qquad\text{and}\qquad 
 T_{\overline{\omega}} S T_{\overline{\omega}} \sim ST_{\overline{\omega}}. 
\end{equation}
The two latter establish the main difference with the case $8\mid D-5$ and are due to the fact that \eqref{om_sq} must be replaced by 
$\omega\overline{\omega}\in 2\mathcal{O}$
and
$\omega^2-\overline{\omega}+1\in 2\mathcal{O}$
when $8\mid D-5$. Multiplying to the left by $T_1$ one deduces (recall $\omega+1-\overline{\omega}\in 2\mathcal{O}$)
\begin{equation}\label {eq:k-even2}
 T_{\overline{\omega}} S T_\omega \sim T_1 S T_\omega
 \qquad\text{and}\qquad 
 T_\omega S T_{\overline{\omega}} \sim T_1 S T_{\overline{\omega}}.
\end{equation}
Hence any word of length at most 3 is equivalent to one of the elements in $C_1\cup C_2\cup \big\{T_1ST_{\omega}, T_1ST_{\overline{\omega}}\big\}$. These elements are inequivalent (this can be obtained essentially following the scheme of the previous proof, where a bigger set is considered).

To conclude that it is a full set of representatives we have to prove that any word of length~4 can be reduced to one of its elements. By \eqref{eq:k-even} and \eqref{eq:k-even2} it is enough to consider 
$ST_1ST_\omega$
and
$ST_1ST_{\overline{\omega}}$.
The following calculation shows 
$ST_1ST_\omega \sim   T_1ST_{\overline{\omega}}$
\[
 (T_1ST_{\overline{\omega}})^{-1}ST_1ST_\omega
 =
 \begin{pmatrix}
  2\overline{\omega}-1
  &
  2\omega\overline{\omega}
  \\
  -2 & 1-2\omega
 \end{pmatrix}
\]
and by conjugation 
$ST_1ST_{\overline{\omega}} \sim   T_1ST_{\omega}$.

\medskip

Summarizing, the relations \eqref{mrel3} were not enough to exclude words of length~4 when $8\mid D-5$ while the two last formulas in \eqref{eq:k-even} cause the collapse of the potential new possibilities in the case~$8\mid D-1$.
}

\begin{proof}[of Proposition~\ref{pvolume}]
 Siegel proved a closed formula \cite[(19)]{siegel2} for $|\Gamma_\mathcal{O}\backslash\mathbb{H}^d|$ when $\mathcal{O}$ is the ring of integers of a totally real number field. In the quadratic case, it reads
 \begin{equation}\label{siegel_v}
  |\Gamma_\mathcal{O}\backslash\mathbb{H}^2|
  =
  \frac{2}{\pi^2}
  \Delta^{3/2}\zeta_D(2)
 \end{equation}
 where $\zeta_D$ is the Dedekind zeta function of $\Q(\sqrt{D})$. We have 
 $\zeta_D(2)=\zeta(2)L(2,\chi)=\frac 16 \pi^2L(2,\chi)$ with $L$ the Dirichlet $L$-function associated to the  character $\chi$.
 By the functional equation of $L$ in its asymmetric form 
 \cite[Cor.10.9]{MoVa}
 (note that $\chi$ is primitive and even)
 \[
  \zeta_D(2)=-\frac{1}{3}\pi^4\Delta^{-3/2}L(-1,\chi).
 \]
 By Proposition~\ref{cosets}, the volume of the fundamental region for $\Gamma$ is~$6-3\chi(2)+6\chi(4)$ times that for  $\Gamma_\mathcal{O}$, hence
 \[
  |\Gamma\backslash\mathbb{H}^2|
  =
  -2\pi^2\big(2-\chi(2)+2\chi(4)\big)L(-1,\chi).
 \]
 Now we appeal to the known formula \cite[Th.4.2]{washington} (as an aside, we note that  the evaluation of $L(1-n,\chi)$ plays an important role in the definition of $p$-adic $L$-functions)
 \[
  L(-1,\chi)
  =
  -\frac{\Delta}{2}
  \sum_{n=1}^\Delta
  \chi(n)B_2(n/\Delta)
  \qquad\text{with}\quad
  B_2(x)=x^2-x+\frac 16.
 \]
 As $\chi$ is even, 
 {\[
  \sum_{n =1}^\Delta n\chi(n)=\sum_{n=1}^\Delta (\Delta-n)\chi(n)
 \]}%
 and { the sum of both sides equal $\Delta \sum_{n=1}^\Delta \chi(n)=0$, hence both vanish}.
 Then we can replace $B_2(n/\Delta)$ by $(n/\Delta)^2$. 
\end{proof}

\section{Proof of the main results}\label{proof_main}

After Lemma~\ref{alemma} and Lemma~\ref{sp_th}, 
the proof of Theorem~\ref{mainth1} and Theorem~\ref{mainth2}  boils down to approximate sharp cuts by smooth kernels and estimating Selberg transforms. Following \cite{chamizo},  with an hyperbolic version of a  classical Euclidean device, it is possible to reduce the whole problem to estimate the Selberg transform of the characteristic function of an interval (which is done in \cite[Lemma~2.4]{chamizo}). 
With this idea in mind we write $\chi_V$ and $h_V$ respectively for the characteristic function of $[0,V]$ and its Selberg transform
\[
 \chi_V(x)=
 \begin{cases}
  1&\text{if }0\le x< V,
  \\
  0&\text{otherwise}
 \end{cases}
 \qquad\text{and}\qquad
 h_V(t)
 =
 \int_{\{u(z,i)\le V\}}
 y^{1/2+it}\; d\mu(z).
\]
The slow decay of $h_V$ would cause serious convergence problems when applying the spectral theorem. The idea introduced in \cite{chamizo} is to replace $\chi_V$ by a  manageable approximation in such a way that its Selberg transform is like a product of two $h_V$'s at different scales. It doubles the decay without requiring  to estimate new transforms.
The key argument is a simple one and it is  summarized in the following result.

\begin{lemma}\label{hyp_conv}
 For $0<v<V$ consider the function $F:\mathbb{H}^2\rightarrow\R$ given by
 \[
  F(z,w)
  =
  \frac{1}{4\pi v}
  \int_{\mathbb{H}}
  \chi_V\big(u(z,\zeta)\big)
  \chi_v\big(u(\zeta,w)\big)
  \; d\mu(\zeta).
 \]
 Then there exists $f:[0,\infty)\rightarrow\R$ such that 
 $F(z,w)=f\big(u(z,w)\big)$
 and 
 \[
  \chi_{V^-}
  \le f\le 
  \chi_{V^+}
  \qquad\text{where}\quad
  V^\pm
  =
  \big(
  \sqrt{V(1+v)}
  \pm
  \sqrt{v(1+V)}
  \big)^2.
 \]
 Moreover the Selberg transform of $f$ is $h_Vh_v$. 
\end{lemma}

\begin{proof}
 Note that $F(\gamma z,\gamma w)=F(z,w)$ for $\gamma\in\text{PSL}_2(\R)$ because $d\mu$ is the invariant measure, then $F(z,w)$ only depends on $\rho(z,w)$ and it assures the existence of $f$. 
 Using geodesic polar coordinates \cite[(1.17)]{iwaniec} it is plain 
 that $\int_{\mathbb{H}}
  \chi_v\big(u(\zeta,w)\big)
  \; d\mu(\zeta)
 =4\pi v$
 (the area of a hyperbolic circle) and hence $0\le f\le 1$. 
 
 Take $\widetilde{V}, \widetilde{v}>0$ such that $2V+1=\cosh \widetilde{V}$ and $2v+1=\cosh\widetilde{v}$. Then by \eqref{rhou} we can write
 \[
  f\big(u(z,i)\big)
  =
  \frac{1}{4\pi v^2}
  \int_{\mathbb{H}}
  \chi_{\widetilde{V}}\big(\rho(z,\zeta)\big)
  \chi_{\widetilde{v}}\big(\rho(\zeta,i)\big)
  \; d\mu(\zeta).
 \]
 By the triangle inequality, if $\rho(z,i)\ge \widetilde{V}+\widetilde{v}$ the integral vanishes. As $0\le f\le 1$, we can rewrite this as
 $f\big((\cosh x-1)/2\big)\le \chi_{\widetilde{V}+\widetilde{v}}(x)$
 for $x\ge 0$ that means $f\le \chi_{V^+}$ 
 with $2V^++1=\cosh(\widetilde{V}+\widetilde{v})$ and the addition formula for $\cosh$ gives the claimed formula for $V^+$.
 In the same way, if $\rho(z,i)\le \widetilde{V}-\widetilde{v}$ then $\rho(\zeta,i)<\widetilde{v}$ implies $\rho(z,\zeta)<\widetilde{V}$ and consequently we can omit $\chi_{\widetilde{V}}$ to get $f\big(u(z,i)\big)=1$. This can be rephrased as 
 $f\big((\cosh x-1)/2\big)\ge \chi_{\widetilde{V}-\widetilde{v}}(x)$
 for $x\ge 0$ and proceeding as before,  $f\ge \chi_{V^-}$.

 The last part of the statement reduces to an application of Fubini's theorem using that 
 \[
  \int_{\mathbb{H}}
  \chi_V\big(u(z,w)\big)
  (\Im z)^{1/2+it}
  \; d\mu(\zeta)
  =
  h_V(t)
  (\Im w)^{1/2+it}.
 \]
 See \cite[Lemma~2.2]{ChRaRu} for the details and a general result.
\end{proof}

Theorem~\ref{mainth1}
and
Theorem~\ref{mainth2}
will be proved by an application of  Lemma~\ref{sp_th}, for 
different choices of the parameters $v_1$ and $v_2$ in the following result.

\begin{proposition}\label{NDker}
 Let $V_j>0$ and $0<v_j<C\min(1,V_j)$, $j=1,2$ where $0<C<1$ is an absolute constant. Then there exists $V_j'>0$ with 
 $V_j'=V_j+O\big(V_jv_j^{1/2}+(V_jv_j)^{1/2}\big)$
 such that
 \[
  N_D(V_1,V_2)
  =
  2
  \sum_{(\gamma,\gamma^\sigma)\in\Gamma}
  k_1\big( u(\gamma i,i) \big)
  k_2\big( u(\gamma^\sigma i,i) \big)
 \]
 where the Selberg transform of $k_j$ is $(4\pi v_j)^{-1}h_{V_j'}h_{v_j}$.
\end{proposition}

{\sc Remark.}
The constant $C$ can be substituted by an explicit numerical value by following the steps in the proof, but this value is unimportant for our arguments, so we shall  not to get into this calculation. 

\begin{proof}
 Let us abbreviate $M_j = \max(1,V_j^{-1/2})$; then $0<C^{-1/2}v_j^{1/2}M_j<1$. 
 
 Take in Lemma~\ref{hyp_conv} $v=v_j$ and $V=V_j(t)$ with 
 $V_j(t)=V_j+tC^{-1/4}V_jM_jv_j^{1/2}$
 and $t\in [-1,1]$. 
 For $t$ in this interval, the hypothesis $0<v<V$ is satisfied choosing $C$ small enough because $C^{-1/4}M_jv_j^{1/2}\le C^{1/4}\to 0$ when $C\to 0^+$. 
 Let $f_{t_j}$ be the corresponding function $f$ in Lemma~\ref{hyp_conv}, which ensures that
 $f_{-1_j}\le\chi_{V_j}\le f_{1_j}$ under
 \[
  (r_j^+-s_j^+)^2\ge V_j
  \quad\text{and}\quad
  (r_j^-+s_j^-)^2\le V_j
  \qquad\text{with}\quad
  \begin{cases}
  r_j^\pm
  =
  \sqrt{V_j(\pm 1)(1+v_j)},
  \\
  s_j^\pm
  =
  \sqrt{v\big(1+V_j(\pm 1)\big)}.
  \end{cases}
 \]
 We have
 \[
  V_j^{-1/2}r_j^\pm
  =
  1\pm \frac 12 C^{-1/4}M_jv_j^{1/2}
  +O\big(
  C^{-1/2}M_j^2v_j
  +v_j
  \big)
  \quad\text{and}\quad
  V_j^{-1/2}s_j^\pm
  =
  O\big(
  M_jv_j^{1/2}
  \big)
 \]
hence the inequalities hold  choosing  $C$ small enough.
 
 Consequently, defining $K_t$ as the automorphic kernel \eqref{aut_ker} that corresponds to $k_t(x,y)=f_{t1}(x)f_{t2}(y)$, we have by Lemma~\ref{alemma}
 \[
  2K_{-1}(\mathbf{i},\mathbf{i})
  \le N_D(V_1,V_2)\le
  2K_{1}(\mathbf{i},\mathbf{i}).
 \]
 The intermediate value theorem implies $N_D(V_1,V_2) = 2 K_{t_0}(V_1,V_2)$ for some $t_0\in [-1,1]$ and the result follows with $V_j'=V_j(t_0)$.
\end{proof}

For the sake of completeness, we include here the estimates we need for the Selberg transform. A more precise result with asymptotic formulas is given in \cite{chamizo}. 

\begin{lemma}
 The Selberg transform $h_V$ of the characteristic function of $[0,V]$ satisfies for $t\ge 0$
 \begin{equation}\label{t_real}
  h_V(t)\ll 
  \begin{cases}
   V^{1/2}(1+t)^{-3/2} &\text{if }t\ge 1, \ V\ge 1
   \\
   V^{1/2}\log(2V) &\text{if }t\le 1, \ V\ge 1
  \end{cases}
  \quad\text{and}\qquad
  h_V(t)\ll 
  V(1+t^2V)^{-3/4} 
  \quad\text{if }V\le 1,
 \end{equation}
 and for $t$ pure imaginary $0<\Im t\le 1/2$ we have
 \begin{equation}\label{t_pure}
  h_V(t)\ll 
   V^{1/2+|t|}
   \min\big(|t|^{-1}, \log(2V)\big)
   \quad\text{if } \ V\ge 1
  \qquad\text{and}\qquad
  h_V(t)\ll 
  V 
  \quad\text{if }V\le 1.
 \end{equation}
\end{lemma}

\begin{proof}
 This is a consequence of Lemma~2.4 in \cite{chamizo}. 
\end{proof}

{ The logarithm appearing in these estimates cannot be avoided when $V$ is large and $t$ close to $0$. It is reflected on some average results for the hyperbolic circle problem  \cite{PhRu}.}

\

Now, we are ready to prove our main results.

\begin{proof}[of Theorem~\ref{mainth1}]
 Take in Proposition~\ref{NDker} $v_1=v_2=(V_1V_2)^{-1/2}$, hence we have
 $V_1'=V_1+O\big(V_1^{3/4}V_2^{-1/4}\big)$
 and
 $V_2'=V_2+O\big(V_2^{3/4}V_1^{-1/4}\big)$.
 With the notation of Lemma~\ref{sp_th},  for $t\in I_n$, \eqref{t_real} gives
 \[
  (4\pi v_j)^{-1}
  \big|h_{V_j'}(t)h_{v_j}(t)\big|
  \ll
  \frac{V_j^{1/2}}{(1+t)^{3/2}(1+t^2(V_1V_2)^{-1/2})^{3/4}}
 \]
 and a term $\log(V_j+1)$ must be introduced for $t\in I_0$. 
 Then $H_j\ll V_j^{1/2}(V_1V_2)^{1/8}$. On the other hand, the $1/9$ bound for $\Im t_{\ell j}$ from \cite{KiSh} implies by \eqref{t_pure}
 \[
  (4\pi v_j)^{-1}
  \big|h_{V_j'}(t_{\ell j})h_{v_j}(t_{\ell j})\big|
  \ll
  V_j^{1/2+1/9}=V_j^{11/18}.
 \]
 Consequently, the sum over $\ell \in \Lambda_0$ in Lemma~\ref{sp_th} contributes
 $O\big((V_1V_2)^{11/18}\big)$ and, as 
 $h_{V_j'}(i/2)=4\pi V_j'$
 and
 $h_{v_j}(i/2)=4\pi v_j$, we obtain the main term \eqref{eq:M} of the automorphic kernel
 \[
  \mathcal{M}
  =
  \frac{16\pi^2 V_1'V_2'}{|\Gamma\backslash\mathbb{H}^2|}
  +
  O\big((V_1V_2)^{11/18}\big)
  =
  \frac{1}{2}C_DV_1V_2
  +
  O\big((V_1V_2)^{3/4}\big)
 \]
 where we have used %the definition 
 expression \eqref{def_N} and Proposition~\ref{pvolume}. 
 On the other hand, the error term in Lemma~\ref{sp_th} is
 \[
  O\big(
  V_1^{1/2}(V_1V_2)^{1/8}
  \cdot
  V_2^{1/2}(V_1V_2)^{1/8}
  +
  V_1^{1/2}(V_1V_2)^{1/8}
  \cdot
  V_2^{11/18}
  +
  V_2^{1/2}(V_1V_2)^{1/8}
  \cdot
  V_1^{11/18}
  \big)
 \]
 and clearly this is $O\big((V_1V_2)^{3/4}\big)$.
\end{proof}

\begin{proof}[of Theorem~\ref{mainth2}]
 Take in Proposition~\ref{NDker} $v_1=(V_1V_2^2)^{-1/2}$ and  $v_2=V_2^{-1/2}$, that implies 
 $V_1'=V_1+O\big(V_1^{1/2}V_2^{-1}\big)$
 and
 $V_2'=V_2+O\big(V_1^{-1/4}V_2^{1/2}\big)$.
 The Selberg transforms of $h_{V_1'}$, $h_{v_1}$ and $h_{v_2}$ admit bounds as in the previous proof that give
 \[
  (4\pi v_1)^{-1}
  \big|h_{V_1'}(t)h_{v_1}(t)\big|
  \ll
  \frac{V_1^{1/2}}{(1+t)^{3/2}(1+t^2(V_1V_2^2)^{-1/2})^{3/4}}
  \qquad\text{for }t\in I_n,\ n\ge 1
 \]
 and  $H_1\ll V_1^{1/2}(V_1V_2^2)^{1/8}=V_1^{5/8}V_2^{1/4}$. 
 
 The difference is that for $h_{V_2'}$ we have to use %employ 
 the second bound in \eqref{t_real} because $V_2'<1$. Then in this case
 \[
  (4\pi v_2)^{-1}
  \big|h_{V_2'}(t)h_{v_2}(t)\big|
  \ll
  \frac{V_2}{(1+t^2V_2)^{3/4}(1+t^2V_1^{-1/2})^{3/4}}
  \qquad\text{for }t\in I_n,\ n\ge 0
 \]
 which leads to $H_2\ll V_1^{1/8}V_2^{1/4}$. 
 
 If there are no exceptional eigenvalues, the error term in Lemma~\ref{sp_th} is simply $O(H_1H_2)=O\big( V_1^{3/4}V_2^{1/2} \big)$ and we get the first part of the result. 

 For the second part, note that we only need a bound for  
 $\Im t_{\ell_ j}$ when $j=1$ because the second formula of \eqref{t_pure} does not involve $t$. Under our assumption 
 \[
  (4\pi v_1)^{-1}
  \big|h_{V_1'}(t_{\ell_ 1})h_{v_1}(t_{\ell_ 1})\big|
  \ll
  V_1^{1/2+c}
  \qquad\text{and}\qquad
  (4\pi v_2)^{-1}
  \big|h_{V_2'}(t_{\ell_ 2})h_{v_2}(t_{\ell_ 2})\big|
  \ll
  V_2.
 \]
 Then the error term in Lemma~\ref{sp_th} includes two new terms following $V_1^{3/4}V_2^{1/2}$, namely, it is
 \[
  O\big(
  V_1^{3/4}V_2^{1/2}
  +
  V_1^{5/8}V_2^{5/4}
  +
  V_1^{5/8+c}V_2^{1/4}
  \big)
  =
  O\big(
  V_1^{3/4}V_2^{1/2}
  +
  V_1^{5/8+c}V_2^{1/4}
  \big).
 \]
 Finally, note that the finite sum in $\mathcal{M}$ contributes $O\big(V_1^{1/2+c}V_2\big)$ that is absorbed by the error term.
\end{proof}

\section{ Appendix. { Numerical considerations}}\label{num_res}

A problem to carry out numerical calculations related to \eqref{N_k} is that $r$ is essentially a divisor function in $\mathcal{O}[i]$ and it seems hard to implement in an efficient way. Here we make some comments for the reader interested in doing direct computations and we display some of the data obtained in this way. 

A simple calculation, already displayed in the proof of Lemma~\ref{conver} shows that for $\lambda= x+y\sqrt{D}\in\mathcal{O}$ the value of $r(\lambda)$ is the  number of solutions $a,b,c,d$ of
\begin{equation}\label{r2}
\begin{cases}
 x = a^2+Db^2+c^2+Dd^2,
 \\
 y = 2ab+2cd.
\end{cases}
\end{equation}
Here, $a+b\sqrt{D}, c+d\sqrt{D}\in\mathcal{O}$. Equivalently, we consider $a,b,c,d\in\Z$ if $D\equiv 1\ (4)$ and we consider  $2a,2b,2c,2d,a-b, c-d\in\Z$ if $D\not\equiv 1\ (4)$.

\medskip

The solutions of \eqref{r2} enjoy some symmetries that are useful for numerical calculations. 

\begin{proposition}\label{pr_symm}
 Given $\lambda =  x+y\sqrt{D}\in\mathcal{O}$, let $M_j$ and $M^*_k$ be the number of solutions $a+b\sqrt{D}, c+d\sqrt{D}\in\mathcal{O}$ of \eqref{r2} satisfying  the mutually exclusive conditions indicated in this list:
 \begin{align*}
  M_1\to&\,0<c<a;           &  M_2\to&\,0=c<a,\ d>0;     &  M_3\to&\,0<c=a,\ d<b;\\
  M_4\to&\,0=c=a,\ 0<d<b;           &  M_1^*\to&\,a=b=0;     &  M_2^*\to&\,a=c,\ b=d.
 \end{align*}
 Then $r(\lambda)=8(M_1+M_2+M_3+M_4)+2(M_1^*+M_2^*)$.
\end{proposition}

\begin{proof}
 We have a group of transformations acting on the solutions, $G\cong\Z_2\times \Z_2\times \Z_2$,  generated by
 \[
 (a,b,c,d)\longmapsto (-a,-b,c,d),
 \quad
 (a,b,c,d)\longmapsto (a,b,-c,-d),
 \quad
 (a,b,c,d)\longmapsto (c,d,a,b).
 \]
 This action is free (fixed point free) except for the set of solutions $S^*$ satisfying one of the following four sets of conditions:
 \[
  a=b=0
  \quad\text{or}\quad
  c=d=0
  \quad\text{or}\quad
  a=c,\ b=d
  \quad\text{or}\quad
  a=-c,\ b=-d.
 \]
 The solutions satisfying the first or the second set of conditions amount $2M_1^*$ and the rest of the conditions give  $2M_2^*$ solutions. In the complement of $S^*$ the action of $G$ gives equivalence classes with eight elements and we can always select exactly one of these eight elements satisfying the conditions indicated for $M_1$, $M_2$, $M_3$ and~$M_4$. This follows imposing an ordering between $0$, $a$ and $c$ when possible i.e., when there are not coincidences among them. In this way $M_1$ counts the ordered case and $M_2$, $M_3$ and $M_4$ the cases with coincidences. 
\end{proof}

For the numerical calculation of $N_D(V_1,V_2)$ note that 
$0\le \lambda<V_1$ 
and 
$0\le \lambda^c<V_2$ 
is equivalent to
\[
 0\le 2x< {V_1+V_2},
 \qquad
 {\max(-x,x-V_2)}
 <y \sqrt{D}<
 {\min(V_1-x,x)}.
\]
The first one readily imposes upper bounds on $|a|$, $|b|$, $|c|$ and $|d|$.
If we let $a$, $b$, $c$ and $d$  vary in these ranges and, for each combination, we add~1 to the position $x$, $y$ of a matrix with $x$ and $y$ as in \eqref{r2}, then this matrix will store the required values of $r(\lambda)$ and the evaluation of $r(\lambda)r(\lambda+1)$ in \eqref{N_k} corresponds to a sum of  products of adjacent entries. 

Now Proposition~\ref{pr_symm} allows to reduce in a factor of approximately $1/8$ the ranges of $a$, $b$, $c$ and $d$ to be considered. In the implementation the loop corresponding to $M_1$ takes the bulk of the calculations because the rest of the loops involve a lesser number of free variables.

The actual issue for direct computations of this kind is the limit of the fast access allocatable memory to store the matrix. The size of the stored elements can be reduced to one fourth noting that $r(\lambda)=r(\lambda^\sigma)$ and that $y$ is always even by \eqref{r2}. With these reductions, when $D\not\equiv 1\ (4)$ in the more demanding case $V_1\approx V_2\approx V$ (less sparse matrix), the memory to be allocated for the values of $x$ and $y$  
is of around
\[
 \sum_{x=0}^V
 \frac{\min(V-x,x)}{2\sqrt{D}}
 \sim \frac{V^2}{8\sqrt{D}}
\]
integers. 
When $V_1$ and $V_2$ are very unbalanced, meaning $V_1$ large and $V_2<1$, then $y$ is determined by $x$ and one only needs room for the around $V_1/2$ integer values of $x$. If $D\equiv 1\ (4)$ these figures should be multiplied by $2$. 

\

We mention here a couple of tables for $D=2$ that we have obtained by implementing the previous idea in an average PC with a simple C program. 

As in the case of the Gauss circle problem, the error term in the asymptotic approximation of $N_2(V,V)$ oscillates and to get a reliable idea on the limits for a general upper bound, rather than picking large special  values of $V$, it is more informative to consider 
\[
 F(x)=\sup_{V<x}|N_2(V,V)-8V^2|.
\]
Extensive computations prove 

\medskip

\begin{center}
\bgroup
\def\arraystretch{0.5}
\setlength\tabcolsep{4pt}
\begin{tabular}{r|c|c|c|c|c|c|c|c|}
\hline
 $x$& 5000& 10000& 15000& 20000& 25000& 30000& 35000& 40000
 \\
 \hline
 $F(x)$& 124508& 383780& 421116& 440700& 882044& 1082980& 1369084& 1807652
 \\
 \hline
\end{tabular}
\egroup
\end{center}

\

As an aside, a possibility that we have not explored is to carry out numerical calculations via Lemma~\ref{alemma} with a description of the group in terms of words and relations (cf. \cite{PhRu}). 
It looks promising when $4\mid D-1$ because $\Gamma_\mathcal{O}$ has simple generators and Proposition~\ref{cosets} gives a full description of the cosets
but when $4\nmid D-1$ the natural generators are linked to the class group \cite[{\rm\S}5.2]{everhart} and the calculations for large values of $D$ could be more difficult.
In both cases one should have some control on $u(\gamma (\mathbf{i}),\mathbf{i})$ in terms of the length of $\gamma$.

\medskip 

In the case  of unbalanced arguments of $N_2$, the numerical experiments suggest that the asymptotic formula 
\[
 N_2(V_1,V_2)\sim 8V_1V_2
\]
in Theorem~\ref{mainth2}
could hold in some ranges beyond the condition   $V_1V_2^2\to\infty$, but even for reasonably large $V_1$ the values of $N_2(V_1,V_2)$ are too small to draw trustworthy conjectures. 
For instance, if we define
\[
 G(V)=\frac 18 V^{-1/2} {N_2(V,V^{-1/2})}
\]
then we have

\medskip

\begin{center}
\bgroup
\def\arraystretch{0.5}
\setlength\tabcolsep{4pt}
\begin{tabular}{r|c|c|c|c|c|}
\hline
 $V$& 10000& 20000& 30000& 40000& 50000
 \\
 \hline
 $N_2(V,V^{-1/2})$& 836& 1220& 1476& 1540 &1924
 \\
 \hline
 $G(V)$& 1.045000&1.078337&1.065211& 0.962500& 1.075548
 \\
 \hline
\end{tabular}
\egroup
\end{center}

\medskip

It is tempting to claim that $G(V)$ goes to $1$ when $V\to \infty$.

\

{We finish with some considerations about the  constant $C_D$ giving the limit of $N_D(V,V)/V^2$}. Here it is a table of some values 

\medskip

\begin{center}
\bgroup
\def\arraystretch{0.5}
\setlength\tabcolsep{4pt}
\begin{tabular}{r|c|c|c|c|c|c|c|c|c|c|}
\hline
 $D$& 2& 3& 5& 6& 7& 101& 1001& 10001& 100001& 1000001
 \\
 \hline
 $C_D$& 8& 4& 8& 4/3& 1& 8/95& 2/753& 1/11616& 4/1462371& 1/11832936
 \\
 \hline
\end{tabular}
\egroup
\end{center}

\medskip

{ As we mentioned before, and the table confirms,  $C_D$   decays as $D^{-3/2}$.}
This decay can be proved theoretically with explicit constants. Namely, we have

\begin{proposition}
 With the notation as in \eqref{def_N}, we have 
 \[
        {\frac{192}{5\Delta^{3/2}}
        <C_D<\frac{240}{\Delta^{3/2}}.}
  \]
\end{proposition}

{This is in accordance with the previous table. In which the minimal value of $\Delta^{3/2}C_D$ is around $84.12$ (reached at $D=1001$) and the maximal value is around $181.02$ (reached at $D=2$).}

\begin{proof}
 The formula \eqref{siegel_v} due to Siegel for the volume of the fundamental region of the full Hilbert modular group and Corollary~\ref{c_index} imply
 \[
  |\Gamma\backslash\mathbb{H}^2|
  =
  \big(2-\chi(2)+2\chi(4)\big)
  \frac{6}{\pi^2}\Delta^{3/2}\zeta(2)L(2,\chi).
 \]
 Clearly
 \[
  \frac{\zeta(4)}{\zeta(2)}
  =
  \prod_p\big(1+p^{-2}\big)^{-1}
  <L(2,\chi)<
  \prod_p\big(1-p^{-2}\big)^{-1}
  =\zeta(2).
 \]
 Using $\zeta(2)=\pi^2/6$, $\zeta(4)=\pi^4/90$ and comparing with Proposition~\ref{pvolume}, 
 \[
  \frac{1}{15}\Delta^{5/2}
  <
  \sum_{n=1}^\Delta
  n^2\chi(n)
  <
  \frac{1}{6}\Delta^{5/2}.
 \]
 Substituting in the definition of $C_D$, we complete the proof. 
\end{proof}

Sharper bounds can be proved separating congruence classes modulo~8. For instance, if $8\mid D-1$ the previous argument leads to
\[
 \Delta^{3/2}C_D=\frac{32\pi^2}{3L(2,\chi)}
 >\frac{32\pi^2}{3\zeta(2)}=64
\]
and in fact $\Delta^{3/2}C_D\to 64$ if we choose $D=4\prod_{p\le N}p+1$ with $N\to\infty$ to force $\chi(p)=1$ for small primes. For example, taking $N=17$ we have  $D= 38798761$ and $\Delta^{3/2}C_D\approx 64.84$.

%\bibliography{bibhilbert}{}
%%\bibliographystyle{plain}
%\bibliographystyle{alpha}

\end{document}